\documentclass[11pt]{amsart}
\usepackage{amsfonts, amsthm, amssymb, amsmath, stmaryrd}
\usepackage{mathrsfs,array}
\usepackage{eucal,color,times,enumerate,accents}
\usepackage[all]{xy}
\usepackage{url}
\usepackage{enumitem}
\usepackage[pdftex,bookmarksnumbered,bookmarksopen]{hyperref}
\hypersetup{
  colorlinks,
  citecolor=black,
  linkcolor=black,
  urlcolor=black}

  \newcommand{\Addresses}{{
  \bigskip
  \footnotesize

\textsc{London School of Geometry and Number Theory, UCL, Department of Mathematics, Gower street, WC1E 6BT, London, UK}\par\nopagebreak
  \textit{E-mail address}, G.~Baldi: \texttt{gregorio.baldi.16@ucl.ac.uk}

}}

\usepackage{tikz-cd}
\usepackage{bbm}

\usetikzlibrary{positioning}
\usepackage{dsfont}
\usepackage{tikz}
\usetikzlibrary{matrix,arrows,decorations.pathmorphing,positioning}
\usepackage{calligra}
\usepackage[colorinlistoftodos]{todonotes}
\usepackage{xy}
\xyoption{all}

\usepackage{tikz}
\usetikzlibrary{matrix,arrows,decorations.pathmorphing}

\usepackage{leftidx}

\input xy
\xyoption{all}

\usepackage[bottom=3cm, top=3cm, left=2.7cm, right=2.7cm]{geometry}

\newcommand{\Fq}{\mathbb{F}_q}

\usepackage{tikz-cd}
\theoremstyle{plain}
\newtheorem{thm}{Theorem}[section]

\newtheorem{conj}[thm]{Conjecture}
 
 \newtheorem{lemma}[thm]{Lemma}
 \newtheorem{prop}[thm]{Proposition}
 
 \newtheorem{rmk}[thm]{Remark}

 \numberwithin{equation}{section}

\DeclareMathOperator{\Hom}{Hom}

\DeclareMathOperator{\tr}{tr}

\newcommand{\Z}{\mathbb{Z}}
\newcommand{\Q}{\mathbb{Q}}

\newcommand{\Oo}{\mathcal{O}}

\newcommand{\F}{\mathcal{F}}

\newcommand{\C}{\mathbb{C}}

\newcommand{\Qbar}{\overline{\mathbb{Q}}}
\newcommand{\Ff}{\mathbb{F}}
\newcommand{\frob}{\operatorname{Frob}}

  \newcommand\blfootnote[1]{%
  \begingroup
  \renewcommand\thefootnote{}\footnote{#1}%
  \addtocounter{footnote}{-1}%
  \endgroup
}

\makeatletter
\def\subtitle#1{\gdef\@subtitle{#1}}
\def\@subtitle{}
\makeatother

\begin{document}

\title{Some remarks on motivical and derived invariants}\blfootnote{\emph{Date}. 10 October, 2019.}
\author{Gregorio Baldi}

\begin{abstract}
We discuss several conjectures about derived equivalent varieties, defined over fields of arbitrary characteristics, and implications among them. In particular we show that the (conjectural) derived invariance of the Hasse--Weil Zeta functions of smooth projective varieties over finite fields, implies the derived invariance of the Hodge diamond of complex algebraic varieties. 
\end{abstract}
\maketitle

\section{Introduction}
Let $K$ be an arbitrary field. We say that $X/K$ is a \emph{nice variety} if it is a smooth projective geometrically irreducible scheme of finite type over $K$. By $\overline{X}$ we denote the base change of $X$ to a fixed algebraic closure $\overline{K}$ of $K$. We study the bounded derived category of coherent sheaves on $X$, denoted by $D^b(X)$. Let $X,Y/K$ be nice varieties, by a \emph{functor} $D^b(X)\to D^b(Y)$ we say a $K$-linear exact functor. We say that two varieties $X,Y/K$ are \emph{derived equivalent} if there exists a functor $D^b(X)\to D^b(Y)$ which is an equivalence of categories. We are interested in \emph{derived invariants}, i.e. properties that derived equivalent nice varieties $X,Y/K$ have to share. For example it is easy to see that the dimension is a derived invariant for nice varieties (see for example \cite[Proposition 4.1]{book} and \cite[Proposition 6.18]{book} where it is proven that the numerical Kodaira dimension is also a derived invariant). 

\subsection{Conjectural derived invaraints}\label{conj}
This short note grew up as an attempt to understand the following far reaching conjecture (\cite[Conjecture 1]{MR2225203}).
\begin{conj}[Orlov]\label{conjmotives}
Derived equivalent nice varieties over a field $K$, have isomorphic $K$-Chow motives.
\end{conj}
We briefly recall how the category of $K$-Chow motives is constructed, for more details we refer to \cite{MR1265529}. Let $\mathcal{V}_K$ be the category of nice $K$-varieties. Given $ X \in \mathcal{V}_K$ and an integer $d$ we denote with $\mathcal{Z}^{d}(X) $ the free abelian group generated by irreducible sub-varieties of $X$ of codimension $d$ and by $\mathcal{A}^d(X)$ the quotient of $\mathcal{Z}^{d}(X)\otimes\Q$ by rational equivalence. Given $X, Y \in \mathcal{V}_K$, we define $\text{Corr}^r(X,Y)$, the group of correspondences of degree $r$ from $X$ to $Y$ as follows. Assuming for simplicity that $X$ is purely $d$-dimensional we set
\begin{displaymath}
\text{Corr}^r(X,Y): = \mathcal{A}^{d+r}(X \times Y).
\end{displaymath}
Thanks to intersection theory, we can define a composition law:
\begin{displaymath}
\text{Corr}^r(X,Y) \times \text{Corr}^s(Y,Z)\to \text{Corr}^{r+s}(X,Z).
\end{displaymath}

We denote with $\text{Corr}_K$ the category whose objects are the objects of $\mathcal{V}_K$ and $\text{Corr}^r(X,Y)$ as the Hom-set between $X$ and $Y$. This is an additive, $\Q$ linear, tensor category, equipped with a tensor functor
\begin{displaymath}
h: \mathcal{V}_K^{\text{op}}\to \text{Corr}_K .
\end{displaymath}
Finally we define the category of $K$-Chow motives as the pseudo-abelian envelope of $\text{Corr}_K$, i.e. objects are triples $(X,i,n)$ where $X$ is an object of $\mathcal{V}_K$, $n$ is an integer and $i$ an idempotent element of $\text{Corr}^0(X,X)$, and
\begin{displaymath}
\Hom_{K-\text{Chow motives}}((X,i,n), (Y,j,m)):=i \  \text{Corr}^{m-n}(X,Y)j \subset \text{Corr}^*(X,Y).
\end{displaymath}
If $X,Y/K$ are nice varieties, we say that they have isomorphic Chow motives, as in Conjecture \ref{conjmotives}, if $h(X):=(X,\text{id}_X,0)$ and $h(Y):=(Y, \text{id}_Y,0)$ are isomorphic in the category of $K$-Chow motives we have just constructed.

\begin{rmk}
In the construction of $K$-Chow motives, it is important to define correspondences as Chow groups tensored with the rational numbers. The reason is that, for example, the derived category of an abelian variety is only an isogeny invariant and so Conjecture \ref{conjmotives} would be false integrally.
\end{rmk}

Even if Orlov's conjecture has no restrictions on the base field $K$, in this note we only consider the case where $K$ is a finite field of characteristic $p$ or a field of characteristic zero. Thanks to the Lefschetz principle, in Algebraic Geometry it is natural to consider only Conjecture \ref{conjmotives} when $K \subset \C$.

From the construction of $K$-Chow motives, we see that if $h(X)$ and $h(Y)$ are isomorphic, then $H^*(X)$ is isomorphic to $H^*(Y)$ for any \emph{reasonable} cohomology theory $H^*$. In particular Conjecture \ref{conjmotives} predicts the following:
\begin{conj}\label{conjzeta}
Derived equivalent nice varieties defined over finite fields have the same Hasse--Weil Zeta function.
\end{conj}

Let $p$ be a prime number and $q$ a power of $p$. Denote by $\Fq$ the finite field with $q$ elements. Given a nice variety $X/\Fq$ its Hasse--Weil Zeta function, denote by $Z(X,t)$, is defined as follows:
\begin{displaymath}
Z(X,t):= \exp \left( \sum _{n\geq 1} \frac{|X(\mathbb{F}_{q^n})|}{n} t^n\right)\in \Q[[t]],
\end{displaymath}
where $|X(\mathbb{F}_{q^n})|$ denotes the number of $\mathbb{F}_{q^n}$-points of $X$. Thanks to the Weil conjectures, proven by Deligne \cite[Theorem 1.6]{weil1}, we know that 
\begin{displaymath}
Z(X,t) = \frac{P_1(t) \dots P_{2\dim X -1}(t)}{P_2(t)  \dots P_{2\dim X }(t)} ,
\end{displaymath}
where each $P_i(t)$ lies in $\Z[t]$, has roots of absolute value $1/\sqrt{q}^i$ and $P_i(t)$ is the characteristic polynomial of the Frobenius automorphism of $H^i(\overline{X},\Q_\ell)$, where $\ell \neq p$ is a prime number.

For the same reason, Conjecture \ref{conjmotives} implies that
\begin{conj}\label{conjhodge}
Derived equivalent nice varieties defined over the complex numbers have the same Hodge numbers.
\end{conj}

\begin{rmk}
Conjecture \ref{conjhodge} is often attributed to Kontsievich and is motivated by its relation with Homological Mirror Symmetry (see the introduction of \cite{MR1403918}). Regarding Conjecture \ref{conjzeta}, to the best of our knowledge, it was considered for the first time for K3 surfaces and odd characteristic by Lieblich and Olsson, see \cite[Theorem 1.2]{MR3429474}.
\end{rmk}

In this note we show how tools coming from Arithmetic Geometry can be used to understand Conjecture \ref{conjhodge}, which, at a first sight, has a more analytical flavour (as it involves the existence of the Hodge filtration). Inspired by Ito's paper \cite{ito}, our main theorem is
\begin{thm}\label{main}
Conjecture \ref{conjzeta} implies Conjecture \ref{conjhodge}.
\end{thm}
Finally we remark that, in analogy with Conjecture \ref{conjzeta}, one can also wonder about what kind properties of rational points of nice varieties defined over number fields are invariant under derived equivalences. For more about this we refer the reader to the introduction of \cite{MR3616008}.
\subsection{Comments on Conjecture \ref{conjzeta}}\label{moreonzeta}
One reason why Conjecture \ref{conjzeta} may be easier than Conjectures \ref{conjhodge} and \ref{conjmotives} is that Zeta functions can be defined without invoking the existence of any cohomology theory. As a result of this many different cohomology theories, from which the Hasse--Weil Zeta function can be read, have emerged. For example:
\begin{itemize}
\item $\ell$-adic \'{e}tale cohomology, \cite[Section 4]{MR292838};
\item $p$-adic cohomology, \cite[Theorem 1]{MR0332791};
\item Topological Hochschild Homology, \cite[Theorem A]{MR3807755}.
\end{itemize}
One could hope that the action of Frobenius endomorphism on (the even and odd parts of) the above cohomology theories, or other cohmological theories yet to be discovered, can be proven to be \emph{intrinsic} so that it can be reconstructed just by looking at the bounded derived category. For a recent use of Topological Hochschild homology to derived invariants in postitive characteristic, we refer to \cite{2019arXiv190612267A}.

It is interesting to see the interplay between Conjecture \ref{conjzeta} and Orlov's Conjecture \ref{conjmotives}. To do so we actually need a weaker version of the conjecture, where Chow motives are replaced by homological or numerical ones. Indeed Conjecture \ref{conjmotives} was stated using Chow motives since the rational equivalence on $\mathcal{Z}^d(X)$ is the finest adequate equivalence relation (as introduced in \cite{MR0116010}). However, in the construction of the category of Chow motives outlined in the previous section, we can define $A^d(X)$ as $\mathcal{Z}^d(X)\otimes \Q$ modulo the homological, resp. numerical, equivalence (once a Weil cohomology is fixed) to obtain the homological, resp. numerical $K$-motives. We will prove

\begin{prop}\label{prop1}
Assume the Tate and semisimplicity conjectures, namely \cite[Conjectures $T^i(X)$ and $SS^i(X)$ for all $i$ and all varieties $X$ over a finite field $\Fq$]{MR1265523}). Then Conjecture \ref{conjzeta} implies that derived equivalent varieties over finite fields have isomorphic homological motives (with rational coefficients).
\end{prop}
Assuming the $\ell$-adic Tate and semisimplicity conjectures for varieties defined over number fields we can also prove the following\footnote{Moonen \cite[Theorem 1]{moontate} has recently proven that, in characteristic zero, the Tate conjecture actually implies the semisimplicity conjecture. To have an analogous statement in positive characteristic one also needs to assume that homological and numerical equivalences agree, as discussed in \cite[Theorem 2]{moontate}.}.
\begin{prop}\label{prop2}
Let $K$ be a number field. Assume the $\ell$-Tate and semisimplicity conjectures, namely \cite[Conjectures $T^j(X)$ and $SS^i(X)$ for all $i$ and all varieties $X$ over $K$]{MR1265523}. Then Conjecture \ref{conjzeta} implies that derived equivalent $K$-varieties have isomorphic homological $K$-motives (with rational coefficients).
\end{prop}
It is not hard to see that the above propositions follow form Deligne's weights theory and the Chebotarev density theorem. Since we were not able to locate such proofs in the literature, we decided to included them for completeness and to stress how arithmetic tools can be useful in understanding Conjecture \ref{conjhodge}.

\subsection{A weaker conjecture}
Proving that the Zeta function is a derived invariant amounts to prove that the number of points, over every finite extension of the base field, does. Recall that coherent sheaves over varieties defined on fields of characteristic $p>0$ have again characteristic $p$. So it may be easier to obtain a mod $p$ invariant out $D^b(X)$. We propose the following special case of Conjecture \ref{conjzeta}.
\begin{conj}\label{conjmodp}
Derived equivalent nice varieties defined over finite fields of characteristic $p$ have the same number, modulo $p$, of rational points.
\end{conj}
Let $X$ be a scheme over a finite field $\Fq$. We denote by $\frob_X$, or simply by $\frob$, the Frobenius endomorphism of $X$, i.e. the unique endomorphism which is the identity on topological spaces and raises each regular function to its $q$th power. It is a finite morphism of schemes. Notice that if $\F$ is an $\Oo_X$ module, then $\F$ and $\frob_* (\F)$ are isomorphic as sheaves of abelian groups and therefore there is a natural isomorphism $H^i(X,\F)\cong H^i(X,\frob_* (\F))$. We implicitly use this in what follows.

Thanks to Fulton's Trace Formula (\cite[page 189]{MR512269}), if $X/\Ff_q$ is a nice variety, we have
\begin{displaymath}
|X(\Ff_q)| \mod p = \sum _{i=0}^{\dim X} (-1)^i \tr(\frob | H^i(X,\Oo_X)).
\end{displaymath}
Compared to the Lefschetz trace formula, which computes the cardinality of $X(\Ff_q)$ from the action of the Frobenius on the $\ell$-adic cohomology ring of $X$, it has the advantage that involves only coherent sheaves. Finally it is important to notice that we do not conjecture that the trace of the Frobenius on each $H^i$ is a derived invariant, but only their sum with alternating signs. Of course the case $i=1$ plays a special role. Indeed we have
\begin{prop}
The number $\tr(\frob | H^1(X,\Oo_X)) \in \Ff_p$ is a derived invariant. More precisely let $X,Y/\Fq$ be derived equivalent nice varieties, then $\tr(\frob_X | H^1(X,\Oo_X))= \tr(\frob_Y | H^1(Y,\Oo_Y))$.
\end{prop}
\begin{proof}
Let $A$ be the Albanese variety of $X$, we have that
\begin{displaymath}
 H^1(X,\Oo_X) \cong H^1(A,\Oo_A) 
\end{displaymath}
and the isomorphism is compatible with the action of the two Frobenii. A simple argument, involving only Galois cohomology, shows that the number of points of an abelian variety depends only on its isogeny class. Since the isogeny class of the Albanese variety is a derived invariant, as established in \cite[Theorem B]{MR3750214}, the result follows.
\end{proof}

\subsection{Known results}
We end the introduction summarising what is known towards the three conjectures presented so far.
Let $X,Y/K$ be derived equivalent nice varieties. They have isomorphic motives when
\begin{itemize}
\item $X$ is a curve of genus different from one, \cite[Corollary 4.13]{book};
\item $X$ has ample or anti-ample canonical bundle, \cite{MR1818984};
\item $X$ is an abelian variety, \cite[Theorem 2.19]{orlovab}, see also \cite[Lemma 2.1]{honigs};
\item $X$ has dimension two, \cite[Theorem 0.1 and Section 2.4]{MR3785793} and \cite[Theorem B]{MR3750214}.
\end{itemize}
Actually in the first two cases, we can conclude that $X$ and $Y$ are isomorphic.

Excluding the cases appearing above we have:
\begin{itemize}
\item Conjecture \ref{conjzeta} has been established for nice varieties of dimension three (\cite[Theorem A]{MR3750214});
\item Conjecture \ref{conjhodge} has been established for nice threefolds (\cite[Corollary 3]{MR2839458}) and if $X$ and $Y$ have dimension $4$ and the same $h^{1,1}$ (\cite[Corollary 1.4]{MR3737326}).
\end{itemize}
To the best of our knowledge, Conjecture \ref{conjmodp} has not been considered before.

\subsection{Acknowledgements}
We thank M.Tamiozzo for stimulating discussions regarding Conjecture \ref{conjzeta} and E.Ambrosi for the interest in this work and numerous discussions about motives. This work was supported by the Engineering and Physical Sciences Research Council [EP/L015234/1], the EPSRC Centre for Doctoral Training in Geometry and Number Theory (The London School of Geometry and Number Theory), University College London.

\section{Proof of the main results}\label{proof}
We first recall a result about derived equivalences and then prove Theorem \ref{main}. Using some classical results about $\ell$-adic cohomology, we prove the two propositions of section \ref{moreonzeta}. We conclude the section reproving Conjecture \ref{conjmotives} for abelian varieties, by looking at the associated Galois representations
\subsection{Orlov's theorem}
Let $K$ be any field, and $X, Y/K$ be nice varieties. Consider $X\times_K Y$ and let $p_X$ (resp. $p_Y$) be the projection onto $X$ (resp. onto $Y$). For background on derived categories we refer to the monograph \cite{book}. Every object $E\in D^b(X\times_K Y)$ defines a \emph{Fourier--Mukai functor} $\Psi_E$ via the assignment:
\begin{displaymath}
\Psi_E : D^b(X)\to D^b(Y), \ \ \ F \mapsto p_{Y *} ({p_X}^* F \otimes E ).
\end{displaymath}
Orlov \cite[Theorem 3.2.1]{MR1998775} proved the following
\begin{thm}[Orlov]\label{orlovthm}
Let $X,Y$ be nice varieties defined over any field $K$ and let $F : D^b(X)\to D^b(Y )$ be an equivalence of categories. There exists an object $E \in D^b(X \times_K Y )$ such that $F$ is isomorphic to the Fourier-Mukai transform $\Psi_E$ and this object is unique up to isomorphism.
\end{thm}
In what follows it is important that the theorem has no restriction on the field $K$.
\subsection{Proof of Theorem \ref{main}}
We first reduce Conjecture \ref{conjhodge} to varieties defined over a number field, thanks to a spreading out argument, and then compute the Hodge diamond from the reduction modulo primes (using Theorem \ref{ito}). 
\begin{prop}\label{propqbar}
It is enough to prove Conjecture \ref{conjhodge} for nice varieties defined over number fields.
\end{prop}
\begin{proof}
Let $X,Y/\C$ be derived equivalent nice varieties, $F:D^b(X)\to D^b(Y)$ an equivalence and $E\in D^b(X\times_\C Y)$ be the unique Fourier-Mukai kernel associated to $F$ (using Theorem \ref{orlovthm}). We claim that:
\begin{itemize}
\item There is variety $S$, defined over a number field $K$, morphisms $X'\to S, Y'\to S$, defined over $K$, such that $X'_0 =X, Y'_0=Y$ for a canonical point $0 \in S(\C)$ (where $X_0$, resp. $Y_0$, denotes the fibre of $X'\to S$, resp. of $Y'\to S$ over $0\in S(\C)$). Moreover there exists an open sub-scheme $U \subset S /K$ such that for all $s\in U(\Qbar)$ the fibres $X_s, Y_s$ are nice varieties and are derived equivalent ($\Qbar$-linearly). 
\end{itemize}
To see this notice that, since their defining equations are given by a finite set of polynomials, $X,Y$ are defined over a sub-field $L$ of $\C$ which is finitely generated over $\Q$. Let $K$ be the number field obtained intersecting $L$ with $\Qbar \subset \C $, and $S/K$ be a model of $L$. To spread out $E$, and the the Fourier-Mukai vector associated to the inverse of $F$, we argue as follows. Assume for simplicity that $E$ is defined by a coherent sheaf $G$ on $X\times_\C Y$ (rather than a bounded complex), and write $G$ as the cokernel of a morphism between free sheaves (of finite rank). Since free sheaves clearly spread out and the morphisms between them involve only finitely many polynomial equations, we see that the locus of points in $S$ whose fibers can fail to be derived equivalent is a closed sub-scheme. Finally, by shrinking $S$, we may also assume $X'$ and $Y'$ are smooth projective $S$-schemes.

Recall that Deligne \cite[Theorem (5.5).]{del} proved that, for a proper smooth family of varieties in characteristic $0$, the Hodge numbers are constants among the fibers. Therefore it is enough to prove that $X'_s$ and $Y'_s$ have the same Hodge numbers for some $s$ in $U(\Qbar)$. This concludes the proposition.
\end{proof}

The last ingredient we will need to prove the theorem is due to Ito, \cite[Proposition 1.2]{ito}. The proof relies on Chebotarev density theorem, $p$-adic Hodge theory \cite{MR924705} and the Weil conjectures \cite{weil1}.
\begin{thm}[Ito]\label{ito}
Let $K$ be a number field and $\Oo_K$ its ring of integers. Let $\mathcal{X}$ and $\mathcal{Y}$ be schemes of finite type over $\Oo_K$ whose generic fibers $X$ and $Y$ are proper and smooth over $K$. If, for all but finitely many maximal ideals $\mathfrak{p}\subset \Oo_K$, we have
\begin{displaymath}
 |\mathcal{X}(\Oo_K/\mathfrak{p})| = |\mathcal{Y}(\Oo_K/\mathfrak{p})|  
\end{displaymath}
then $X$ and $Y$ have equal Hodge numbers.
\end{thm}

We are now ready to prove the main theorem.
\begin{proof}[Proof of Theorem \ref{main}]
Thanks to Proposition \ref{propqbar}, we may assume that $X$ and $Y$ are defined over a number field $K$. Denote by $\Oo_K$ the ring of integers of $K$. We may find schemes of finite type $\mathcal{X},\mathcal{Y}$ defined over $\Oo_K$ with generic fibre $X$ and $Y$ that are smooth outside a finite set of places. We claim that for all but finitely many maximal ideals $\mathfrak{p}\subset \Oo_K$ the reductions mod $\mathfrak{p}$ of $X$ and $Y$ are again derived equivalent (as $\Oo_K/ \mathfrak{p}$-varieties). The argument to prove this is similar to the one presented in Proposition \ref{propqbar} to spread out complex of coherent sheaves. Indeed the definitions of the objects $E\in D^b(X\times_K Y)$ and $E'\in D^b(Y\times_K X)$ corresponding to a derived equivalence between $X$ and $Y$ and its inverse involve only finitely many denominators and so determine, for all but finitely many primes $\mathfrak{p}$, an object $E_\mathfrak{p} \in D^b((\mathcal{X}\times_{\Oo_K}\mathcal{Y})\times_{\Oo_K} \Oo_K / \mathfrak{p})$ which realises the desired equivalence (with inverse $E'_{\mathfrak{p}}$).

Thanks to Conjecture \ref{conjzeta}, we have that the reductions mod $\mathfrak{p}$ of $X$ and $Y$ have the same number of points (for all but finitely many primes $\mathfrak{p}\subset \Oo_K$). Eventually we may apply Theorem \ref{ito} to conclude that $X$ and $Y$ have equal Hodge numbers.
\end{proof}

\subsection{Proof of Proposition \ref{prop1}}
The following argument is certainly well known to the experts. Since we were not able to locate it in the literature we offer a sketch of the following, which implies Proposition \ref{prop1}. 

\begin{lemma}\label{lemma1}
Let $X,Y/\Fq$ be nice varieties defined. Assume that, for some prime $\ell \neq p$, the Tate and semisimplicity conjectures for the $\ell$-adic cohomology of $X\times _{\Fq} Y$ are known (in all possible codimensions). If $Z(X,t)=Z(Y,t)$, then $X$ and $Y$ have isomorphic homological $K$-motives.
\end{lemma}

\begin{proof}
If $Z(X,t)=Z(Y,t)$, then, by Lefschetz fixed point theorem (see for example \cite[Theorem 3.1]{MR463174}), the set of eigenvalues of the geometric Frobenius (on $X$) acting on 
\begin{displaymath}
\bigoplus_{i=0}^{\dim X} (-1)^i H^i(\overline{X},\Q_\ell)
\end{displaymath}
is equal to the set of eigenvalues of the geometric Frobenius (on $Y$) acting on 
\begin{displaymath}
\bigoplus_{i=0}^{\dim Y} (-1)^i H^i(\overline{Y},\Q_\ell).
\end{displaymath}
Invoking the Weil conjectures (Deligne's theorem \cite{weil1}) we now see that $\dim X = \dim Y$ and that eigenvalues of the Frobenii acting on $H^i(\overline{X},\Q_\ell)$ and $H^i(\overline{Y},\Q_\ell)$ are the same, for every $i=0,\dots, \dim X=\dim Y$. Since the Galois representations we are considering are assumed to be semisimple, the Brauer--Nesbitt theorem shows that there exists an isomorphism 
\begin{displaymath}
H^i(\overline{X},\Q_\ell) \cong H^i(\overline{Y},\Q_\ell)
\end{displaymath}
commuting with the action of the absolute Galois group of $\Fq$. The Tate conjecture now implies that $X$ and $Y$ have isomorphic motives (with $\Q$-coefficients, not only $\Q_\ell$-coefficients). Indeed it is not hard to see that the $\ell$-adic Tate conjecture can equivalently be formulated as saying that the $\ell$-adic realisation functor from the category of (homological) motives to the category of Galois representation is fully faithful. 
\end{proof}
\subsection{Proof of Proposition \ref{prop2}}
Arguing as in the proof of Theorem \ref{main}, to prove Proposition \ref{prop2} it is enough to prove the following. Such statement can be thought as a stronger, but conjectural, version of Theorem \ref{ito}. Here we really use that $K$ is a number field, indeed the same statement is not true if $K$ is a finite field.
\begin{lemma}
Let $K$ be a number field and $\Oo_K$ its ring of integers. Let $\mathcal{X}$ and $\mathcal{Y}$ be schemes of finite type over $\Oo_K$ whose generic fibers $X$ and $Y$ are proper and smooth over $K$. If for all but finitely many primes $\mathfrak{p}\subset \Oo_K$ we have
\begin{displaymath}
 |\mathcal{X}(\Oo_K/\mathfrak{p})| = |\mathcal{Y}(\Oo_K/\mathfrak{p})|  ,
\end{displaymath}
then for every prime $\ell$, up to semisimplification, there exists an isomorphism
\begin{displaymath}
H^*(\overline{X},\Q_\ell)\cong H^*(\overline{Y},\Q_\ell)
\end{displaymath}
compatible with the action of the absolute Galois group of $K$. Therefore the Tate and semisimplicity conjectures for $X\times_K Y$ predict that $X$ and $Y$ have $K$-homological motives.
\end{lemma}
Notice that, in characteristic zero, the Tate conjecture predicts the agreement of numerical and homological equivalence (see for example \cite[5.4.2.2]{MR2115000}).

\begin{proof}
To prove the result we may use Chebotarev density theorem (see for example \cite[Theorem I.2.3]{serrebook}), which asserts that the Frobenii at $\mathfrak{p}$ are dense in the absolute Galois group of $K$. As in the proof of \cite[Proposition 1.2]{ito}, the lemma follows from the Weil conjectures and the smooth and proper base change theorems.
\end{proof}

\subsection{A remark on abelian varieties}
We conclude showing how the Tate conjecture for abelian varieties, as proven by Tate over finite fields and Faltings over number fields, can be used to show that derived equivalence implies an isogeny relation, confirming Conjecture \ref{conjmotives}.

Let $A,B/K$ be derived equivalent abelian varieties. In \cite{orlovab}, Orlov's first step to show that $A$ and $B$ are isogenous is to replace $K$ by one of its algebraic closures. We sketch how to reprove this result using the point of view we adopted in this note when $K$ is (the algebraic closure of) a finite field or a number field. Let $\ell$ a prime different from the characteristic of $K$. Thanks to the celebrated Isogeny theorem, it is enough to prove that there exists an isomorphism
\begin{displaymath}
H^1(\overline{A},\Q_\ell) \cong H^1(\overline{B},\Q_\ell)
\end{displaymath}
compatible with the Galois action. As observed by Honigs \cite[Lemma 3.1.]{honigs}, we have an isomorphism, compatible with the Frobenius action:
\begin{displaymath}
\bigoplus_{i= 1, \dots, \dim A} H^{2i-1}(\overline{A},\Q_\ell)(i-1) \cong \bigoplus_{i=1,\dots ,\dim A = \dim B} H^{2i-1}(\overline{B},\Q_\ell)(i-1) .
\end{displaymath}
Since the higher cohomological groups of an abelian variety are given by symmetric powers of the $H^1$, we deduce that $H^1(\overline{A},\Q_\ell) \cong H^1(\overline{B},\Q_\ell)$, proving the claim.

\bibliographystyle{alpha}
\bibliography{biblio.bib}

\begin{thebibliography}{Moo19}

\bibitem[AB19]{2019arXiv190612267A}
Benjamin {Antieau} and Daniel {Bragg}.
\newblock {Derived invariants from topological Hochschild homology}.
\newblock {\em arXiv e-prints}, page arXiv:1906.12267, Jun 2019.

\bibitem[Abu17]{MR3737326}
Roland Abuaf.
\newblock Homological units.
\newblock {\em Int. Math. Res. Not. IMRN}, (22):6943--6960, 2017.

\bibitem[And04]{MR2115000}
Yves Andr\'{e}.
\newblock {\em Une introduction aux motifs (motifs purs, motifs mixtes,
  p\'{e}riodes)}, volume~17 of {\em Panoramas et Synth\`eses [Panoramas and
  Syntheses]}.
\newblock Soci\'{e}t\'{e} Math\'{e}matique de France, Paris, 2004.

\bibitem[BO01]{MR1818984}
Alexei Bondal and Dmitri Orlov.
\newblock Reconstruction of a variety from the derived category and groups of
  autoequivalences.
\newblock {\em Compositio Math.}, 125(3):327--344, 2001.

\bibitem[Del68]{del}
Pierre Deligne.
\newblock Th\'{e}or\`eme de {L}efschetz et crit\`eres de
  d\'{e}g\'{e}n\'{e}rescence de suites spectrales.
\newblock {\em Inst. Hautes \'{E}tudes Sci. Publ. Math.}, (35):259--278, 1968.

\bibitem[Del74]{weil1}
Pierre Deligne.
\newblock La conjecture de {W}eil. {I}.
\newblock {\em Inst. Hautes \'{E}tudes Sci. Publ. Math.}, (43):273--307, 1974.

\bibitem[Del77]{MR463174}
Pierre Deligne.
\newblock {\em Cohomologie \'{e}tale}, volume 569 of {\em Lecture Notes in
  Mathematics}.
\newblock Springer-Verlag, Berlin, 1977.
\newblock S\'{e}minaire de g\'{e}om\'{e}trie alg\'{e}brique du Bois-Marie SGA
  $4\frac{1}{2}$.

\bibitem[Fal88]{MR924705}
Gerd Faltings.
\newblock {$p$}-adic {H}odge theory.
\newblock {\em J. Amer. Math. Soc.}, 1(1):255--299, 1988.

\bibitem[Ful78]{MR512269}
William Fulton.
\newblock A fixed point formula for varieties over finite fields.
\newblock {\em Math. Scand.}, 42(2):189--196, 1978.

\bibitem[Hes18]{MR3807755}
Lars Hesselholt.
\newblock Topological {H}ochschild homology and the {H}asse-{W}eil zeta
  function.
\newblock In {\em An alpine bouquet of algebraic topology}, volume 708 of {\em
  Contemp. Math.}, pages 157--180. Amer. Math. Soc., Providence, RI, 2018.

\bibitem[Hon15]{honigs}
Katrina Honigs.
\newblock Derived equivalent surfaces and abelian varieties, and their zeta
  functions.
\newblock {\em Proc. Amer. Math. Soc.}, 143(10):4161--4166, 2015.

\bibitem[Hon18]{MR3750214}
Katrina Honigs.
\newblock Derived equivalence, {A}lbanese varieties, and the zeta functions of
  3-dimensional varieties.
\newblock {\em Proc. Amer. Math. Soc.}, 146(3):1005--1013, 2018.
\newblock With an appendix by Jeffrey D. Achter, Sebastian Casalaina-Martin,
  Katrina Honigs, and Charles Vial.

\bibitem[HT17]{MR3616008}
Brendan Hassett and Yuri Tschinkel.
\newblock Rational points on {K}3 surfaces and derived equivalence.
\newblock In {\em Brauer groups and obstruction problems}, volume 320 of {\em
  Progr. Math.}, pages 87--113. Birkh\"{a}user/Springer, Cham, 2017.

\bibitem[Huy06]{book}
D.~Huybrechts.
\newblock {\em Fourier-{M}ukai transforms in algebraic geometry}.
\newblock Oxford Mathematical Monographs. The Clarendon Press, Oxford
  University Press, Oxford, 2006.

\bibitem[Huy18]{MR3785793}
D.~Huybrechts.
\newblock Motives of derived equivalent {K}3 surfaces.
\newblock {\em Abh. Math. Semin. Univ. Hambg.}, 88(1):201--207, 2018.

\bibitem[Ito03]{ito}
Tetsushi Ito.
\newblock Birational smooth minimal models have equal {H}odge numbers in all
  dimensions.
\newblock In {\em Calabi-{Y}au varieties and mirror symmetry ({T}oronto, {ON},
  2001)}, volume~38 of {\em Fields Inst. Commun.}, pages 183--194. Amer. Math.
  Soc., Providence, RI, 2003.

\bibitem[Kle68]{MR292838}
S.~L. Kleiman.
\newblock Algebraic cycles and the {W}eil conjectures.
\newblock In {\em Dix expos\'{e}s sur la cohomologie des sch\'{e}mas}, volume~3
  of {\em Adv. Stud. Pure Math.}, pages 359--386. North-Holland, Amsterdam,
  1968.

\bibitem[KM74]{MR0332791}
Nicholas~M. Katz and William Messing.
\newblock Some consequences of the {R}iemann hypothesis for varieties over
  finite fields.
\newblock {\em Invent. Math.}, 23:73--77, 1974.

\bibitem[Kon95]{MR1403918}
Maxim Kontsevich.
\newblock Homological algebra of mirror symmetry.
\newblock In {\em Proceedings of the {I}nternational {C}ongress of
  {M}athematicians, {V}ol. 1, 2 ({Z}\"{u}rich, 1994)}, pages 120--139.
  Birkh\"{a}user, Basel, 1995.

\bibitem[LO15]{MR3429474}
Max Lieblich and Martin Olsson.
\newblock Fourier-{M}ukai partners of {K}3 surfaces in positive characteristic.
\newblock {\em Ann. Sci. \'{E}c. Norm. Sup\'{e}r. (4)}, 48(5):1001--1033, 2015.

\bibitem[Moo19]{moontate}
Ben Moonen.
\newblock A remark on the {T}ate conjecture.
\newblock {\em J. Algebraic Geom.}, 28(3):599--603, 2019.

\bibitem[Orl02]{orlovab}
D.~O. Orlov.
\newblock Derived categories of coherent sheaves on abelian varieties and
  equivalences between them.
\newblock {\em Izv. Ross. Akad. Nauk Ser. Mat.}, 66(3):131--158, 2002.

\bibitem[Orl03]{MR1998775}
D.~O. Orlov.
\newblock Derived categories of coherent sheaves and equivalences between them.
\newblock {\em Uspekhi Mat. Nauk}, 58(3(351)):89--172, 2003.

\bibitem[Orl05]{MR2225203}
D.~O. Orlov.
\newblock Derived categories of coherent sheaves, and motives.
\newblock {\em Uspekhi Mat. Nauk}, 60(6(366)):231--232, 2005.

\bibitem[PS11]{MR2839458}
Mihnea Popa and Christian Schnell.
\newblock Derived invariance of the number of holomorphic 1-forms and vector
  fields.
\newblock {\em Ann. Sci. \'{E}c. Norm. Sup\'{e}r. (4)}, 44(3):527--536, 2011.

\bibitem[Sam60]{MR0116010}
Pierre Samuel.
\newblock Relations d'\'{e}quivalence en g\'{e}om\'{e}trie alg\'{e}brique.
\newblock In {\em Proc. {I}nternat. {C}ongress {M}ath. 1958}, pages 470--487.
  Cambridge Univ. Press, New York, 1960.

\bibitem[Sch94]{MR1265529}
A.~J. Scholl.
\newblock Classical motives.
\newblock In {\em Motives ({S}eattle, {WA}, 1991)}, volume~55 of {\em Proc.
  Sympos. Pure Math.}, pages 163--187. Amer. Math. Soc., Providence, RI, 1994.

\bibitem[Ser89]{serrebook}
Jean-Pierre Serre.
\newblock {\em Abelian {$l$}-adic representations and elliptic curves}.
\newblock Advanced Book Classics. Addison-Wesley Publishing Company, Advanced
  Book Program, Redwood City, CA, second edition, 1989.
\newblock With the collaboration of Willem Kuyk and John Labute.

\bibitem[Tat94]{MR1265523}
John Tate.
\newblock Conjectures on algebraic cycles in {$l$}-adic cohomology.
\newblock In {\em Motives ({S}eattle, {WA}, 1991)}, volume~55 of {\em Proc.
  Sympos. Pure Math.}, pages 71--83. Amer. Math. Soc., Providence, RI, 1994.

\end{thebibliography}

\Addresses

\end{document}